\documentclass[10pt]{article}
\usepackage{amsmath, amsthm, graphicx, amssymb, color, url}

\newtheorem{theorem}{Theorem}[section]

\newtheorem{corollary}[theorem]{Corollary}
\newtheorem{lemma}[theorem]{Lemma}

\input xy
\xyoption{all}

\DeclareMathOperator{\GL}{GL}

\DeclareMathOperator{\SL}{SL}
\DeclareMathOperator{\Aut}{Aut}

\begin{document}
\bibliographystyle{plain}

\title{Quantifying Residual Finiteness}

\author{Khalid Bou-Rabee}
\maketitle

\begin{abstract}
We introduce the notion of quantifying 
the extent to which a finitely generated group is residually finite. We investigate this behavior for examples that include  free groups, the first Grigorchuk group, 
finitely generated nilpotent groups, and certain arithmetic groups such as $\SL_n( \mathbb{Z})$. 
In the context of finite nilpotent quotients, we find a new characterization of finitely generated nilpotent groups.
\end{abstract}

\section*{Introduction}

Given a finitely generated group $G$, one natural question that has attracted interest is the asymptotic growth of the number of subgroups of $G$ of index $n$. Indeed, the subject of subgroup growth predates the study of word growth (see Lubotzky and Segal \cite{lubsegal-2003}, page xvi). 
In this context, the class of residually finite groups is particularly interesting as they have a rich collection of finite 
index subgroups. Recall that such groups have the property that the intersection of all finite index subgroups is trivial. Given this property, one might ask how quickly this intersection becomes trivial or in the same 
vein, how well finite quotients of $G$ approximate $G$. The goal of this article is to make precise and investigate this question for several classes of residually finite groups. Before stating our main results on this, some notation is required. 

For a fixed finite generating set $S$ of $G$ and $g \in G$, let $\| g \|_S$ denote the word length of $g$ with respect to $S$.
Define
$$k_G (g) := \min \{ |Q| \::\: \text{ $Q$ is a finite quotient of $G$ where $g \neq 1$}\},$$
and 
 $$F_G^S(n) := \max\{  k_G (g) \;:\;\| g \|_S \leq n \}.$$
The objective of this paper is to study the asymptotic properties of $F_G^S$ and one of its variants.

Throughout the paper, we write $f_1 \preceq f_2$ to mean that there exists a $C$ such that $f_1(n) \leq C f_2(Cn)$ for all $n$, and we write $f_1 \simeq f_2$ to mean $f_1 \preceq f_2$ and $f_2 \preceq f_1$.
The dependence of $F_G$ on the generating set is mild, a fact that we will see in Section 1.
Consequently, we will suppress the dependence of $F_G$ on $S$ for the remainder of the introduction.
In that same section, we will also provide some general facts on the behavior of $F_G$ under group extensions, passage to subgroups, and taking direct products.

Our first main result establishes the polynomial growth of $F_G$ for certain arithmetic lattices---see Section 2 for the definition of $\mathcal{O}_L$.
\begin{theorem}
\label{lineargrouptheorem}
Let $L$ be a finite extension of $\mathbb{Q}$. If $k \geq 2$, then $F_{\SL_k( \mathcal{O}_L)}(n) \preceq n^{k^2-1}$.
Moreover, if $k > 2$, then $F_{\SL_k( \mathcal{O}_L)}(n) \succeq n$.
 \end{theorem}

Notice that the asymptotic upper bound for $F_G$ depends only on the dimension of the algebraic group $\SL_n(\mathbb{C})$ and not on the field $L$.
Further, since $\mathbb{Z}*\mathbb{Z} \leq SL_2( \mathbb{Z})$, we have, as a consequence of Theorem \ref{lineargrouptheorem} and Lemma \ref{groupextlemma}, that $F_{F}(n) \preceq n^3$ for any finitely generated non-abelian free group $F$.
The author, unfortunately, does not know of a sharper upper bound for the growth of $F_F(n)$.

There are examples of groups with sub-polynomial and super-polynomial $F_G$ growth.
Let the \emph{Hirsch number} of $G$, denoted $h(G)$, be the number of infinite cyclic factors in a series for $G$ with cyclic or finite factors.
In Section 3 we find a general bound for nilpotent groups of a given Hirsch number:

\begin{theorem} \label{nilpotentgrouptheorem}
Let $P$ be a finitely generated nilpotent group.
Then $F_{P}(n) \preceq \log(n)^{h(P)}$.
\end{theorem}
\noindent
In Section \ref{grigsection}, we present some calculations that show that the first Grigorchuk group has exponential $F_G$ growth.

In the last section we restrict our attention to finite nilpotent quotients and the asymptotic growth of the associated function for these quotients. To be precise, let
$$k_G^{nil}(g) = \min \{ |Q| \;:\;\text{ $Q$ is a finite nilpotent quotient of $G$ where $g \neq 1$}\}$$
and
$$
F_G^{nil}(n) = \max \{ k_G^{nil}(g) \;:\;\| g \| \leq n \}.
$$
Then we get the following characterization of finitely generated nilpotent groups in Section 5.

\begin{theorem} \label{bigtheorem}
Let $G$ be any finitely generated group.
Then $F_G^{nil}(n)$ has growth which is polynomial in $\log(n)$ if and only if
$G$ is nilpotent.
\end{theorem}

The ingredients used in the proofs of the above theorems include the prime number theorem, Cebotar\"ev's Density Theorem, the Strong Approximation Theorem, the congruence subgroup property of $\SL_k (\mathbb{Z})$ for $k > 2$, and Mal'cev's representation theorem for nilpotent groups.

\subsubsection*{Acknowledgements}
I would like to extend special thanks to my advisor, Benson Farb for suggesting this topic, providing encouragement, sound advice, good teaching, and many good ideas.
I am also especially grateful to my coadvisor, Ben McReynolds, for his time, good ideas, and comments on the paper.
I would also like to acknowledge 
Tom Church, 
Matt Day, 
Asaf Hadari, 
Justin Malestein, 
Aaron Marcus, 
Shmuel Weinberger, and
Thomas Zamojski for reading early drafts of this paper and for stimulating conversations on this work.
Finally, it is my pleasure to thank the excellent referee for making comments, suggestions, and a number of corrections on an earlier draft of this paper.

\section{Basic Theory}

In this first section, we lay out some basic lemmas for the sequel.
Recall that a group $G$ is \emph{residually finite} if for any nontrivial $g$ in $G$ there exists a finite group $Q$ and a homomorphism $\psi : G \to Q$ such that $g \notin \ker \psi$.
We begin with a lemma that when applied twice with $G = H$, will let us drop the decoration $S$ in $F_G^S(n)$.

\begin{lemma} \label{groupextlemma}
Let $G$ and $H \leq G$ be residually finite groups finitely generated by $S$ and $L$ respectively.
Then $F_H^L(n) \preceq  F_G^S(n)$.
\end{lemma}

\begin{proof}
As any homomorphism of $G$ to $Q$ restricts to a homomorphism of $H$ to $Q$, it follows that $k_H(h) \leq k_G(h)$ for all $h \in H$.
Hence,
\begin{equation} \label{firstlemmaeq1}
F_H^L(n) = \sup\{  k_H (g) \: :\: \| g \|_L \leq n \} \leq 
\sup\{  k_G (g) \:: \: \| g \|_L \leq n \}.
\end{equation}

\noindent
Further, there exists a $C>0$ such that any element in $L$ can be written in terms of at most $C$ elements of $S$.
Thus,

\begin{equation} \label{firstlemmaeq2}
\{ h \in H : \| h \|_L \leq n \} \subseteq \{ g \in G : \| g \|_S \leq Cn \}.
\end{equation}

\noindent
So by (\ref{firstlemmaeq1}) and (\ref{firstlemmaeq2}), we have that
$$
F_H^L(n) \leq 
\sup\{  k_G (g) : \| g \|_L \leq n \}
\leq 
\sup\{  k_G (g) :  \| g \|_S \leq C n \}
= F_G^S(C n),$$
as desired.
\end{proof}

\noindent
The previous lemma implies that the growth functions for all non-abelian finitely generated free groups are equivalent.
The next lemma shows that $F_G$ is well behaved under direct products.
We leave the proof as an exercise to the reader, as it is straightforward.

\begin{lemma} \label{prodbehavior}
Let $G$ and $H$ be residually finite groups generated by finite sets $S$ and $T$ respectively.
Then
$$\max \{ F_G^S(n), F_H^T(n) \} = F_{G \times H}^J (n),$$
where $J = S \times T$.
\end{lemma}

\noindent
The next lemma shows that growth under finite group extensions is moderately well-behaved. We leave the proof as an exercise to the reader, as it is also straightforward.

\begin{lemma}
\label{finiteindexsubgrouplemma}
Let $H \leq G$ be two finitely generated groups with $[G:H]< \infty$.
Then $F_G(n) \preceq (F_H (n))^{[G:H]}$.
\end{lemma}

\noindent
\section{Arithmetic groups}

In order to quantify residual finiteness for arithmetic groups, we require some auxiliary results concerning the ring analogue of growth for rings of algebraic integers $\mathcal{O}_L$.

\subsection{The integers}

Fix the generating set $\{ 1 \}$ for the integers $\mathbb{Z}$.
For $\mathbb{Z}$ we can do much better than the obvious bound $F_\mathbb{Z}(n) \leq n+1$.
In fact, the elements with the largest value of $k_\mathbb{Z}$ are of the form $\psi(r):= lcm(1, \ldots, r)$.


\begin{lemma} \label{middlelemma}
If $l_1 < \psi(m) < l_2 < \psi(m+1)$, then 
$k_\mathbb{Z}(\psi(m))$ is greater than or equal to $k_\mathbb{Z}(l_1)$ and $k_\mathbb{Z}(l_2)$.
\end{lemma}

\begin{proof}
We prove this by induction on $m$.
The base case with $\psi(2)=2$ and $\psi(3)=6$ are easily checked.
For the inductive step, suppose the statement is true for $m < n$, and let
$l_1 < \psi(m+1) < l_2 < \psi(m+2)$.
By the inductive hypothesis, and the fact that $k_\mathbb{Z}(\psi(\cdot))$ is nondecreasing, we deduce that $k_\mathbb{Z}(\psi(m+1)) \geq k_\mathbb{Z}(l_1)$.
In order for $k_\mathbb{Z}(\psi(m+1)) < k_\mathbb{Z}(l_2)$ we must have that $l_2$ satisfies $j | l_2$ for all $j=1, \ldots, m+2$.
Thus $l_2$ is a multiple of $1, \ldots, m+2$ and thus $l_2 \ge \psi(m+2)$, which is absurd.
\end{proof}


The function $\psi(x)$ is well-studied, as the asymptotic behavior of $\psi$ is used to prove the prime number theorem in analytic number theory.
In fact (see Proposition 2.1, page 189, in Stein and Shakarchi \cite{steinshak}),
\begin{equation} \label{complexresult}
\lim\limits_{x \to \infty} \frac{\log(\psi(x))}{x} = 1.
\end{equation}
Since Lemma \ref{middlelemma} shows that $F_\mathbb{Z}$ and $\psi$ are related, it is no surprise that (\ref{complexresult}) is used to prove the following theorem.

\begin{theorem} \label{Zcase}
We have $F_\mathbb{Z}(n) \simeq \log(n)$.
\end{theorem}

\begin{proof}
Lemma \ref{middlelemma} gives
$$
\frac{F_\mathbb{Z}(n)}{\log(n)} = \frac{k_\mathbb{Z} (\psi(k_n))}{\log(n)}
$$
where $k_n$ is the maximum value of $m$ with $\psi(m) \leq n$.
Since $\log$ is increasing we have 
\begin{equation}
\frac{k_\mathbb{Z}(\psi(k_n))}{\log(\psi(k_n + 1))} \leq
\frac{k_\mathbb{Z}(\psi(k_n))}{\log(n)} \leq
\frac{k_\mathbb{Z}(\psi(k_n))}{\log(\psi(k_n))}. \label{longineq}
\end{equation}

The left hand side of (\ref{longineq}) with $k_\mathbb{Z}(\psi(k_n)) \geq k_n+1$ and (\ref{complexresult}) gives
$$\lim_{n\to \infty} \frac{k_\mathbb{Z} (\psi(k_n))}{\log(n)} \geq 1.$$
The right hand side of (\ref{longineq}) with $k_\mathbb{Z}(\psi(k_n)) \leq  2 k_n$ and (\ref{complexresult}) gives
$$\lim_{n\to \infty} \frac{k_\mathbb{Z} (\psi(k_n))}{\log(n)} \leq 2.$$
Thus $F_\mathbb{Z}(n) \simeq \log(n)$ as desired.
\end{proof}

\begin{corollary}
We have $F_{\mathbb{Z}^d}(n) \simeq \log(n)$.
\end{corollary}

\begin{proof}
This follows immediately from Lemma \ref{prodbehavior} and Theorem \ref{Zcase}.
\end{proof}

\subsection{Rings of Integers}

Let $L/\mathbb{Q}$ be a finite extension, and let $\mathcal{O}_L$ be the ring of integers.
With these conditions and some work, it can be shown that $\mathcal{O}_L$ is a residually finite ring and a finitely generated abelian group.
We need to define $F_{\mathcal{O}_L}$ while keeping the ring structure of $\mathcal{O}_L$ in mind, because $\SL_n(-)$ is a functor from the category of rings to the category of groups.
Equip $\mathcal{O}_L$ with a word metric as a finitely generated abelian group and define
$$k_{\mathcal{O}_L} (g) := \min \{ |Q| \::\: \psi(g) \neq 1, \psi : \mathcal{O}_L \to Q\},$$
 where the maps $\psi$ are ring homomorphisms, and 
 $$F_{\mathcal{O}_L}(n) := \max\{  k_{\mathcal{O}_L} (g) \;:\;\| g \| \leq n \}.$$
The obvious analogue of Lemma \ref{groupextlemma} holds for $F_{\mathcal{O}_L}(n)$.

To study the asymptotic behavior of $F_{\mathcal{O}_L}(n)$, we need some algebraic number theory.
If $\mathfrak{p}$ is a prime ideal of $\mathbb{Z}$, then $\mathfrak{p} \mathcal{O}_L$ is an ideal of $\mathcal{O}_L$ and has factorization
$$
\mathfrak{p}\mathcal{O}_L = \mathfrak{p}_1^{e_1} \cdots \mathfrak{p}_c^{e_c}
$$
where $\mathfrak{p}_i$ are distinct.
Let $f_{\mathfrak{p}_i}$ be the degree of the field extension $[\mathcal{O}_L/\mathfrak{p}_i:\mathbb{Z}/\mathfrak{p}]$.
If $e_i = 1$ and $f_{\mathfrak{p}_i} = 1$ for all $i$, we say that $\mathfrak{p}$ \emph{splits} in $\mathcal{O}_L$.
In the case where $\mathfrak{p} = (p)$ where $p$ is a prime number in $\mathbb{Z}$, we have that $\mathfrak{p}$ splits only if each prime $\mathfrak{p}_i$ that appears in the factorization satisfies $\mathcal{O}_L/\mathfrak{p}_i = p$.
Thus, the primes $(p)$ that split are nice in that they then give small quotients for $\mathcal{O}_L$.
These nice primes appear quite often.
Indeed, the Cebotar\"ev Density Theorem (see Theorem 11, page 414 in Lubotzky and Segal \cite{lubsegal-2003}) implies that the natural density of such primes is non-zero in the set of all primes.
This implication, along with the prime number theorem, gives the following Theorem:

\begin{theorem} \label{OLbound}
We have $F_{\mathcal{O}_L}(n) \simeq \log(n)$.
\end{theorem}

\begin{proof}
The lower bound, $F_{\mathcal{O}_L}(n) \succeq \log(n)$, follows from Lemma \ref{groupextlemma} for rings and Theorem \ref{Zcase}.

For the upper bound, the main idea is that we will first use the bound for $\mathbb{Z}$ to ensure that one of the coordinates in an integral basis for $\mathcal{O}_L$ will not vanish in a small quotient, then we will use the Cebotar\"ev Density Theorem find an even smaller quotient where the element does not vanish.
Let $S = \{b_1, \ldots, b_k \}$ be an integral basis for $\mathcal{O}_L$, and fix a nontrivial $g$ in $\mathcal{O}_L$ with $\|g \|_S = n$.
Then $g = \sum_{i=1}^n a_i b_i$ where $a_i \in \mathbb{Z}$ and $|a_i| \leq n$.
Since $g\neq 0$ there exists  $k$ such that $a_k \neq 0$.
By the Cebotar\"ev Density Theorem, the natural density of the set $P$ of all primes in $\mathbb{Z}$ that split over $\mathcal{O}_L$ has nonzero natural density in the set of all primes in $\mathbb{Z}$.
We claim that there exists $C > 0$, which does not depend on $n$, and a prime $q$ such that $(q)$ splits over $\mathcal{O}_L$ and $q \leq C \log(n)$ and $a_k \not\equiv 0 \mod q$.
Indeed, enumerate $P = \{ q_1, q_2, \ldots  \}$.  
Let $q_{r+1}$ be the first prime in $P$ such that $a_k \not \equiv 0 \mod q_{r+1}$.
Then $q_1 \cdots q_{r}$ divides $a_k$ and by the prime number theorem and positive density of $P$, we have that $q_{r+1} \leq M r \log(r)$ for some $M > 0$, depending only on $L$.
A simple calculation shows that there exists $M'>0$ such that $q_1 \cdots q_r \geq e^{M' r\log(r)}$.
Hence, $q_{r+1} \leq C \log(a_k)$, where $C > 0$ depends only on $L$. The claim is shown.

Hence, we have that $(q) = \mathfrak{q}_1 \cdots \mathfrak{q}_{c}$ with
$|\mathcal{O}_L/\mathfrak{q}_i| = q.$
Further, since $q$ does not divide $a_k$ and since the integral basis $S$ gets sent to a basis in $\mathcal{O}_L/(q)$, we have that $g \neq 1$ in $\mathcal{O}_L/(q)$.
Hence, there exists one $\mathfrak{q}_i$ with $g \neq 1$ in $\mathcal{O}_L/\mathfrak{q}_i$.
As the cardinality of $\mathcal{O}_L/\mathfrak{q}_i$ is equal to $q$ which is no greater than $C \log(n)$, we have the desired upper bound. 
\end{proof}

\subsection{Proof of Theorem \ref{lineargrouptheorem}}

Let $L/\mathbb{Q}$ be a finite extension, and let $\mathcal{O}_L$ be the ring of integers.
With the results of the previous section, we can now obtain results for $\SL_k( \mathcal{O}_L)$.
Note that $\SL_k(\mathcal{O}_L)$ is finitely generated, but this fact is nontrivial (see Platonov and Rapinchuk \cite{platonov}, Chapter 4). 

\begin{theorem}\label{upperslbound} If $k \geq 2$, we have $F_{\SL_k(\mathcal{O}_L)}(n) \preceq n^{k^2-1}$.
\end{theorem}

\begin{proof} 
The strategy in this proof is to bound the entries of a word of length $n$ in $\SL_k( \mathcal{O}_L)$ and then to use this bound to approximate the group using the $F_{\mathcal{O}_L}$ result.
Let $A_1, \ldots, A_r$ be generators for $\SL_k( \mathcal{O}_L)$.
Let $S$ be an integral basis for $\mathcal{O}_L$.
Let $g \in \SL_k( \mathcal{O}_L)$ be a nontrivial element with word length less than or equal to $n$.
It is straightforward to see that there exists a $\lambda > 0$, depending only on $A_1, \ldots, A_r,$ and $S$, such that the following holds: there exists an off-diagonal nonzero entry, $a \neq 0$, or a diagonal entry $a \neq 1$ of the matrix $g$ that has $\|a\|_S \leq \lambda^n$.
For simplicity we assume that $a$ is an off-diagonal entry, a similar argument to what we will give works otherwise.
Since $a$ is in $\mathcal{O}_L$, Theorem \ref{OLbound} gives a $D> 0$, depending only on $L$ and $S$, and a ring homomorphism $\psi : \mathcal{O}_L \to \mathbb{Z}/d\mathbb{Z}$ where $d < D (\log(\lambda^n))$ such that $a \notin \ker \psi$.
The function $\psi$ induces a map $\psi' : \SL_k( \mathcal{O}_L) \to \SL_k( \mathbb{Z}/d \mathbb{Z})$ where $g \notin \ker \psi'$.
By our bound for $d$ we have
$$d^{k^2-1} \leq D^{k^2-1}( \log(\lambda^n))^{k^2-1} \leq D^{k^2-1}( \log(\lambda))^{k^2-1} n^{k^2-1},$$
giving $F_{\SL_k( \mathcal{O}_L)} (n) \preceq n^{k^2-1}$ as asserted.
\end{proof}

For the next theorem we need to introduce some definitions concerning certain subgroups of $\SL_k(\mathbb{Z}/n\mathbb{Z})$.
A normal subgroup of $\SL_k( \mathbb{Z}/n\mathbb{Z})$ is said to be a \emph{principal congruence subgroup} if it is the kernel of some map $\varphi: \SL_k( \mathbb{Z}/n\mathbb{Z}) \to \SL_k( \mathbb{Z}/d\mathbb{Z})$ induced from the natural ring homomorphism $\mathbb{Z}/n \mathbb{Z} \to \mathbb{Z}/d \mathbb{Z}$, where $d$ properly divides $n$.
A subgroup in $\SL_k( \mathbb{Z}/n\mathbb{Z})$ which contains some principal congruence subgroup is said to be a \emph{congruence subgroup}.
A subgroup in $\SL_k( \mathbb{Z}/n\mathbb{Z})$ which does not contain any principal congruence subgroups is said to be \emph{essential}.

\begin{theorem} \label{lowerslbound}
If $k > 2$, then $F_{\SL_k(\mathcal{O}_L)} (n) \succeq n$.
\end{theorem}

\begin{proof}
We first pick a candidate for the lower bound.
By Lubotzky-Mozes-Raghunathan \cite{LMR}, Theorem A, there exists a finite generating set, $S$, for $\SL_k( \mathbb{Z})$ (see also Riley \cite{riley-2004}) and a $C > 0$ satisfying
$$
\| - \|_S \leq C \log(\| - \|_1),
$$
where $\| - \|_1$ is the $1$-operator norm for matrices.
Thus, as $\log(\| E_{ij} (\psi(n)) \|_1) \simeq \log(\psi(n)) \simeq n$, the elementary matrix $E_{ij}(\psi(n))$
may be written in terms of at most $C n$ elements from $S$.
This elementary matrix is our candidate.

Now we show that our candidate, $E_{ij}(\psi(n))$, takes on the lower bound.
Suppose that $Q$ is the finite quotient with the smallest cardinality such that $E_{ij}(\psi(n))$ does not vanish.
Since $\SL_k( \mathbb{Z})$ has the congruence subgroup property (see Bass-Lazard-Serre \cite{serre}), then for the map $\delta: \SL_k( \mathbb{Z}) \to Q$, we have an integer $d$ such that the following diagram commutes.
$$\xymatrix{
\SL_k( \mathbb{Z}) \ar[dr]^{} \ar[r]^{\delta} & Q  \\
& \SL_k( \mathbb{Z}/d \mathbb{Z}) \ar[u]^{}
}$$
Hence, by our choice of $E_{ij}(\psi(n))$, we have that $Q = \SL_k( \mathbb{Z}/d \mathbb{Z})/N$ where $d \geq n$.
In the case where $N$ is a principal congruence subgroup, we see that the smallest finite quotient of $\SL_k( \mathbb{Z})$ where $E_{ij}(\psi(n))$ is nontrivial has size greater than $|\SL_k( \mathbb{Z}/ n \mathbb{Z})| \succeq n$.
If $N$ is a congruence group, then taking the quotient by the largest principal congruence group $N$ contains reduces to the case of $N$ being an essential group.
Then by Proposition 6.1.1 in Lubotzky and Segal \cite{lubsegal-2003}, we have that there exists a $c > 0$, depending only on $k$, such that $|Q| \geq c n$.
Since $\SL_k( \mathbb{Z})$ is contained in $\SL_k( \mathcal{O}_L)$, Lemma \ref{groupextlemma} gives the claim.
\end{proof}

\section{Proof of Theorem \ref{nilpotentgrouptheorem}}

In the proof of Theorem \ref{nilpotentgrouptheorem}, we require the following lemma.

\begin{lemma} \label{uppertricase}
Let $U$ be the group of $d \times d$ integral upper triangular unipotent matrices.
If $G \leq U$, then $F_G(n)  \leq C \log(n)^{h(G)}$, where $C$ does not depend on $n$.
\end{lemma}

\begin{proof}
It is easy to see that entries of matrices of word length $n$ in $G$ are bounded by $C n^r$ for some fixed $r$.
Take $g \in G$.
Then Theorem \ref{Zcase} gives some $D > 0$, which does not depend on $n$, such that $p \leq D \log(n)$ and the natural map $\psi: U \to U_p$ has $g \notin \ker \psi$, where $U_p$ is the image of $U$ in $\GL_d( \mathbb{Z}/p\mathbb{Z})$ consisting of unipotent upper triangular matrices.
So long as $p$ is greater than $d$, we have that $U_p$ has exponent $p$.
Thus $|G| \leq p^{h(G)}$, giving $|G| \leq D^{h(G)} \log(n)^{h(G)}$.
Setting $C = D^{h(G)}$ finishes the proof.
\end{proof}

In the following proof we will reduce the general case to the case in the previous lemma.

\begin{proof} [Proof of Theorem \ref{nilpotentgrouptheorem}]
To start, we may assume, without loss of generality, that $G$ is a torsion-free, finitely generated nilpotent group.
By Mal'cev's Theorem (See Segal \cite{Segal}, Chapter 5, \S B, Theorem 2, or Hall \cite{edmonton:hall}, p. 56, Theorem 7.5)
there exists a canonical injective homomorphism $\beta_N : G \to U$, where $U$ is a group of $d \times d$ integral upper triangular unipotent matrices.
Hence, the bound given by Lemma \ref{uppertricase} finishes the proof.
\end{proof}

\section{The first Grigorchuk group} \label{grigsection}

Let $T$ be the collection of finite sequences of 1s and 0s of length $n \geq 0$.
We will be interested in the automorphisms of $T$ defined inductively by:
\begin{eqnarray*}
a(\xi_1, \ldots, \xi_n) &=& (\overline{\xi_1}, \xi_2, \ldots, \xi_n)  \\
b(0, \xi_1, \ldots, \xi_n) &=&  (0, \overline{\xi_1}, \ldots, \xi_n)) \\
b(1, \xi_1, \ldots, \xi_n) &=&  (1, c(\xi_1, \ldots, \xi_n)) \\
c(0, \xi_1, \ldots, \xi_n) &=&  (0, \overline{\xi_1}, \ldots, \xi_n) \\
c(1, \xi_1, \ldots, \xi_n) &=&  (1, d(\xi_1, \ldots, \xi_n)) \\
d(0, \xi_1, \ldots, \xi_n) &=&  (0, \xi_1, \ldots, \xi_n) \\
d(1, \xi_1, \ldots, \xi_n) &=&  (1, b(\xi_1, \ldots, \xi_n)) \\
\end{eqnarray*}
where $\overline{1} = 0$ and $\overline{0} = 1$.
Let the \emph{first Grigorchuk Group} be $\Gamma:=\left< a, b, c, d \right>$ as in Grigorchuk \cite{Grigo-80} (see also de la Harpe \cite{harpe-2000}, Chapter VIII).
In the case when $g \in \Aut(T)$ fixes the first $k$ entries of any element in $T$, we will write
$g = (\gamma_1, \ldots, \gamma_{2^k})_k$ in order to record the action beyond level $k$ only.
In this case, we say that $g$ has \emph{level $k$}.
For example, $b = (a, c)_1$, $c = (a, d)_1$, and $d = (1, b)_1$ all have level 1.

Let $T(k)$ be the collection of sequences of length at most $k$.
The truncation $T \to T(k)$ induces a map $\psi_k : \Gamma \to \Aut(T(k))$; a \emph{principal congruence subgroup} is equal to $\ker \psi_k$ for some $k$.
Let $\Gamma_k$ be the image of $\psi_k$ in $\Aut(T(k))$.
We borrow from de la Harpe \cite{harpe-2000}, page 238:

\begin{lemma} \label{cong}
For $k \geq 3$, $|\Gamma_k| = 2^{5 \cdot 2^{k-3} + 2}$.
\end{lemma}

\begin{theorem}
We have $F_\Gamma(n) \preceq 2^n$. \label{grigupperbound}
\end{theorem}

\begin{proof}
Let $g$ be an element of word length $\leq n$.
We will show that there exists a $C > 0$, not depending on $g$, such that $k_\Gamma(g) \leq C 2^n$.
To this end, we claim that there exists $k \leq \log(n)$ such that at level $k$ there is an odd number of $a$ symbols appearing in some coordinate of $g$.
Hence $\psi_k(g) \neq 1$.
Suppose $g$ is in reduced word form.
Then the relations $cb = bc = d$ and $dc = cd = b$ give that $g$ must be in a form conjugate to
$g = a e_1 a e_2 a \cdots e_k r$ where $e_i \in \{ b, c, d, b^{-1}, c^{-1}, d^{-1} \}$ and $r \in \{ 1, a \}$.
If $r = a$, then we are done, as $\psi_1(g) \neq 1$.
Otherwise, we see that the number of the symbols describing $g$ on some coordinate of level $2$ is nonzero and no greater than $(|g| + 1)/2$.
And so, by induction, the number of symbols describing $g$ on some coordinate of level $k+1$ is nonzero and no greater than $(((|g| + 1)/2 + 1)/2 + \cdots + 1)/2
= |g| 2^{-k} + 2^{-1} + 2^{-2} + \cdots + 2^{-k} = 2^{-k} |g| + \left( 1 - 2^{-k}\right)$.
Hence, we see that there is some $k$ with $2^{-k} |g| \geq 1$, such that some coordinate of level $k+1$ has an odd number of $a$ symbols.
And so $\psi_{k+2}(g) \neq 1$ where $|g| \geq 2^{k}$, giving some $C > 0$ such that $k_\Gamma(g) \leq C 2^n$ by Lemma \ref{cong}.
\end{proof}

\begin{lemma} \label{elementexists}
There exists a $C > 0$ such that the element $(1, \ldots, 1, (ab)^2)_k$ is in $\Gamma$ and has word length less than $C 2^k$.
\end{lemma}

\begin{proof} 
We will prove this by induction on $k$.
For the base case, observe that
 $$(ab)_0^2 d^{-1} (ab)_0^{-2} d = (abad)_0^2 = (c,a)_1 (1, b)_1 (c, a)_1 (1, b)_1 = (1, (ab)^2)_1.$$
For the inductive step, let $g_k = (1, 1, \ldots,1,  (ab)^2)_k$.
Then conjugating $g_k$ by one of $b$, $c$ or $d$ yields $(1, 1, \ldots, 1, (ab)^2)_{k+1}$.
\end{proof}

\begin{theorem} We have $F_\Gamma(n) \succeq 2^n$. \label{griglowerbound}
\end{theorem}

\begin{proof}
By Lemma \ref{elementexists} there exists a nontrivial element $g \in \Gamma$ of word length no greater than $C 2^n$ such that any $k < n$ has $\psi_k(g) = 1$.
Let $N$ be the normal subgroup of $\Gamma$ of smallest index such that $g\notin N$.
If any element in $N$ has level $k$, then by the proof of Theorem 42 on page 239 in de la Harpe \cite{harpe-2000}, $N$ must contain $\ker \psi_{k+6}$.
Hence, as $g \notin N$, the normal subgroup $N$ must act trivially on the first $n-6$ levels of the rooted binary tree.
Thus, $N$ is contained in $\ker \psi_{n-6}$ and so has index greater than or equal to $2^{5\cdot 2^{n-9} + 2}$ when $n \geq 9$, by Lemma \ref{cong}, giving the desired lower bound.
\end{proof} 




\section{Proof of Theorem \ref{bigtheorem}}

Theorem \ref{bigtheorem} follows from Lemma \ref{secondlemma}, below, and Theorem \ref{nilpotentgrouptheorem}.

\begin{lemma} \label{secondlemma}
If $G$ is a finitely generated group that is not nilpotent, then
$n \preceq F_G^{nil}(n) $.
\end{lemma}

\begin{proof}
Let $S$ be a finite set of generators for $G$.
It suffices to show that there exists a $C >0$ such that for any $n$, there exists $g$ with $\| g \|_S \leq C 2^n$ and $k_G^{nil}(g) \geq 2^n$.
Fix $n > 0$, then since $G$ is not nilpotent, $\Gamma_n(G) \neq 1$.
Recall that $\Gamma_n(G)$ is normally generated by elements of the form
$
[a_1, \ldots, a_n]
$
where $a_i \in S$ or $a_i^{-1} \in S$ for every $i$.
Since $\Gamma_n(G) \neq 1$, there exists some element $[a_1, \ldots, a_n]$ as above that is nontrivial.
Hence, there exists some $g$ with $\| g \|_S \leq C 2^n$ with $g \in \Gamma_n(G)$.
Any finite nilpotent quotient $Q$ where $g \neq 1$ must be nilpotent of class $n+1$ or more giving $|Q| \geq 2^n$.
Thus $k_G^{nil}(g) \geq 2^n$, as desired.
\end{proof}

\end{document}